\documentclass[11pt]{article}
\usepackage[top=1in, bottom=1in, left=1in,right=1in]{geometry}
\usepackage{amsmath,
            amssymb,
            titlesec,
            amsthm,
            mathrsfs,
            theoremref,
            stix,
            hyperref,
            soul,
            graphicx}

\usepackage[all]{xy}
\usepackage{tikz-cd}

\newtheorem{defn}{\textsc{Definition}}[section]
\newtheorem{thm}[defn]{\textsc{Theorem}}
\newtheorem{cor}[defn]{\textsc{Corollary}}
\newtheorem{lem}[defn]{\textsc{Lemma}}
\newtheorem{prop}[defn]{\textsc{Proposition}}
\newtheorem*{rmk}{\textsc{Remark}}
\newtheorem*{eg}{\textsc{Example}}
\newtheorem*{thm*}{\textsc{Theorem}}

\newcommand{\etale}{étale }
\newcommand{\Gal}{\operatorname{Gal}}

\newcommand{\Frob}{\operatorname{Frob}}
\newcommand{\Spec}{\operatorname{Spec}}

\title{Splitting of Polynomial Families via Galois Theory}
\author{Tianhao Wang}
\date{\today}

\begin{document}
\maketitle
\begin{abstract}
We study the splitting behavior of parametrized families of polynomials over finite fields through a geometric and Galois-theoretic approach. While the underlying techniques are widely considered folklore in arithmetic geometry, they have rarely been written down explicitly. To maximize accessibility, we develop a framework based on classical Galois theory and the Chebotarev Density Theorem over an affine normal variety, avoiding the heavy machinery of Grothendieck's \'etale topology.

The primary goal is to extend and conceptually explain a recent result by Slavov, which established the condition for square values of several polynomials over a finite field to be independent. In the case where $q\equiv 1 \pmod n$, we generalize this phenomenon to $n$-th power residues, and reframe this independence condition as the natural condition on Kummer extensions to be mutually linearly disjoint.  Finally, we briefly mention how these results can be translated into the modern language of \'etale fundamental groups, generalizing the base to geometrically integral, normal schemes of finite type over $\mathbb{F}_q$.
\end{abstract}

\tableofcontents
\newpage

\section{Introduction and Main Results}\label{sec:intro}
The distribution of splitting types of polynomials over finite fields has a long history originating in classical algebraic number theory and arithmetic geometry. The Dedekind-Kummer theorem relates polynomial factorization modulo a prime to the splitting of prime ideals. Building on this, Frobenius explicitly connected it to Galois theory via the Frobenius element, proving that the density of primes yielding a specific splitting type is governed by the frequency of the corresponding cycle type in the Galois group. This framework culminated in the Chebotarev Density Theorem, which refined Frobenius's work by establishing the equidistribution of individual Frobenius conjugacy classes. 

The primary goal of this paper is to extend and conceptually explain a recent result by Slavov, which established the condition for square values of several polynomials over a finite field to be independent \cite[Theorem~3]{slavov2024square}. In the case where $q\equiv 1 \mod n$, We generalize this phenomenon to $n$-th power residues. More importantly, we reframe this independence as the natural condition on Kummer extensions to be mutually linearly disjoint. We will prove the following theorem. 

\begin{thm*}[Theorem \ref{thm:slavov_generalization}]
Let $X$ be an affine normal variety over $\mathbb{F}_q$. Let $n \ge 2$, and assume $q \equiv 1 \pmod n$. For rational functions $f_1, \dots, f_m \in \mathbb{F}_q(X)^\times$, if the only relation of the form
$$\lambda \prod_{i=1}^m f_i^{\varepsilon_i} \in (K^\times)^n, \qquad \lambda \in \mathbb{F}_q^\times, \quad 0 \le \varepsilon_i < n$$
is the trivial one where $\varepsilon_i=0$ for all $i$ and $\lambda$ is a perfect $n$-th power in $\mathbb{F}_q^\times$,  then the values of the $n$-th power characters $(\chi_n(f_1(x_0)), \ldots, \chi_n(f_m(x_0)))$ are asymptotically independent and uniformly distributed.

More precisely, Let $U\subset X$ be the open dense subset where all $f_i$ are regular and non-zero. For any prescribed pattern of $n$-th power characters $g \in \mu_n^m$, the number of rational points $x_0 \in U(\mathbb{F}_q)$ evaluating to this character pattern is given by:
$$\frac{|\{ x_0 \in U(\mathbb{F}_q) \mid (\chi_n(f_1(x_0)), \dots, \chi_n(f_m(x_0))) = g \}|}{|U(\mathbb{F}_q)| } = \frac{1}{n^m} + O(q^{-1/2}).$$
\end{thm*}

As will be explained in Remark \ref{rmk: can drop affine and normal condtion}, the above theorem also holds even if we drop the affine and normal condition. 

\subsection*{Outline of the Paper}

We build the framework of studying the splitting of polynomial family in an accessible way, and avoid Grothendieck's \'etale language. In Section \ref{sec:splitting_I} and Section \ref{sec:splitting_II}, we introduce the necessary theory on the splitting of polynomials families using Galois Theory, Frobenius at an unramified place and Chebotarev's density theorem. In Section \ref{sec:simultaneous}, we move to simultaneous splittings of several polynomial families, proving that their splitting statistics are asymptotically independent when the corresponding Galois extensions are mutually linearly disjoint. This is a generalization of Cohen's classic work \cite[Section~3]{Cohen1972distributionII} to a broader parameter family. We will then use the result to generalize a recent work by Slavov. Finally, we reformulate our framework using Grothendieck's \'etale language, where we further generalize the base space from an affine normal variety to a geometrically integral, normal scheme of finite type over $\mathbb{F}_q$.

We should also emphasize that every variety $X/\mathbb{F}_q$ in this paper is assumed to be geometrically integral. This ensures that $\overline{\mathbb{F}_q}\cap K(X) = \mathbb{F}_q$ where $K(X)$ is the function field of $X$. 

\newpage

\section{Splitting of Polynomials via Galois Theory I}\label{sec:splitting_I}
\subsection*{Main statement}
We will study the splitting behavior of a one-parameter family of polynomials as the parameter changes. We refer to \emph{Algebraic Number Theory} by Neukirch \cite[Section~II.7]{neukirch1999algebraic} and \emph{Number Fields} by Marcus \cite[Chapter~3-4]{marcus1977number} for the theory of prime decomposition in function fields and number fields.

We will prove a Galois analogue of the Dedekind-Kummer Theorem. We use the following setup:
\begin{itemize}
    \item $\mathbb{F}_q$: a finite field of $q$ elements.
    \item $K = \mathbb{F}_q(x)$: the function field of $\mathbb{A}^1/\mathbb{F}_q$.
    \item $O_K = \mathbb{F}_q[x]$: the ring of integers in $K$, which is also the coordinate ring of $\mathbb{A}^1/\mathbb{F}_q$.
    \item $p(x, t)\in O_K[t]$: a monic irreducible separable polynomial of degree $n$ in the variable $t$.
    \item $r_1,\ldots, r_n\in \overline{K}$: the $n$ distinct roots of $p$ in $\overline{K}$.
    \item $K(r_i)/K$: the degree $n$ field extension of $K$ obtained by adjoining one root of $p$.
    \item $L/K$: the splitting field of $p(x,t)$ over $K$. We also have $L = K(r_1,\ldots, r_n).$
    \item $O_L$: the ring of integers in $L$.
    \item $G = \Gal(L/K)$: the Galois group of $L/K$.
\end{itemize}

Since $G$ acts faithfully on the roots $\{r_1,\ldots, r_n\}$, it induces
an embedding $G\hookrightarrow S_n$. We will therefore view elements of $G$ as permutations of the roots, and speak of their \textbf{cycle types}. We also say that the polynomial $p(x_0, t)$ has \textbf{splitting type} $(f_1,\ldots, f_r)$ if $p(x_0, t)$ is a product of $r$ irreducible polynomials over $\mathbb{F}_q$ of degrees $f_1,\ldots, f_r$ respectively.

We have the following theorem.
\begin{thm}[Frobenius]\label{thm:spliting of polynomials via Galois theory I}
Let $x_0\in\mathbb{A}^1(\mathbb{F}_q)$ be a point such that $p(x_0, t)$ has distinct roots in $\overline{\mathbb{F}_q}$. Then:
\begin{enumerate}
\item The prime $(x-x_0)$ is unramified in $L$.
\item Let $\Frob_{x_0}\in G$ be the Frobenius associated to $(x-x_0)$. Then the cycle type of $\Frob_{x_0}$ coincides with the splitting type of $p(x_0,t)$ over $\mathbb{F}_q$.
\item If $p(x_0, t)$ has splitting type $(f_1,\ldots,f_r)$, then
$$d=\mathrm{lcm}(f_1,\ldots,f_r),$$
and the splitting field of $p(x_0, t)$ over $\mathbb{F}_q$ is $\mathbb{F}_{q^d}$.
\item The prime $(x-x_0)$ splits in the ring of integers of $K(r_i)$ as
    $$\mathfrak{p}_1\cdots\mathfrak{p}_r,$$
    where each $\mathfrak{p}_i$ has inertia degree $f_i$ over $(x-x_0)$.
\end{enumerate}
\end{thm}
\begin{proof}
    Assume that the prime ideal $(x-x_0)$ factors in $O_L$ as
    $$(x-x_0)O_L=\mathfrak{P}_1^e\cdots \mathfrak{P}_r^e,$$
    where the $\mathfrak{P}_i$ are primes of $O_L$ lying above $(x-x_0)$. We consider the natural embedding
    $$O_K/(x-x_0) \hookrightarrow O_L/\mathfrak{P}_i\cong \mathbb{F}_{q^d},$$
    where $d$ is the inertia degree of $\mathfrak{P}_i$ over $(x-x_0)$.
    Since each root $r_i\in L$ satisfies a monic polynomial with coefficients in $O_K$, we have $r_i\in O_L$. We may reduce the factorization $p(x_0,t)=(t-r_1)\cdots(t-r_n)$
    modulo $\mathfrak{P}_i$ to obtain
    $$ p(x_0,t) = (t-\overline{r}_1)\cdots(t-\overline{r}_n),$$
    where $\overline{r}_i = r_i \bmod \mathfrak{P}_i$ is an element of $\mathbb{F}_{q^d}$.

    Let $D$ and $E$ denote the decomposition group and inertia group of $\mathfrak{P}_i$ over $(x-x_0)$. For $\sigma\in E$, we have $\sigma(r_i)\equiv r_i \pmod{\mathfrak{P}_i}$. Since $\sigma$ permutes the roots, we must have $\sigma(r_i)=r_j$ for some $j$. Reducing this modulo $\mathfrak{P}_i$ gives $\overline{r}_i=\overline{r}_j$. By assumption, the roots of $p(x_0, t)$ are distinct, and we must have $i=j$. Hence $\sigma$ acts trivially on the roots $r_1,\ldots, r_n$ and therefore $\sigma=\mathrm{id}$. Therefore, the inertia group is trivial and $(x-x_0)$ is unramified.

    Let $\overline{G}=\Gal(\mathbb{F}_{q^d}/\mathbb{F}_q)$. Since the inertia group is trivial, we have $D \cong \overline{G}$ via the natural reduction map. Let $\sigma_q$ denote the Frobenius automorphism where $\sigma_q(a)=a^q$. The corresponding element in $D$ is the Frobenius element $\Frob_{x_0}\in G$.

    The action of $\Frob_{x_0}$ on each root $r_i$ is determined by the action of $\sigma_q$ on $\overline{r}_i$. Consequently, the permutation of the roots induced by $\Frob_{x_0}$ has the same cycle type as the action of $\sigma_q$ on the reduced roots $\{\overline{r}_1, \ldots, \overline{r}_n\}$. If the splitting type of $p(x_0, t)$ over $\mathbb{F}_q$ is $(f_1,\ldots,f_r)$, then the cycle type of $\sigma_q$ is $(f_1,\ldots, f_r)$. Since $\Frob_{x_0}$ has the same cycle type as $\sigma_q$, its order is $\mathrm{lcm}(f_1,\ldots,f_r)$. Since the decomposition group $D$ is generated by $\Frob_{x_0}$, we obtain
    $$d=|D|=\mathrm{lcm}(f_1,\ldots,f_r).$$

    Finally, since $p(x_0, t)$ splits completely in $\mathbb{F}_{q^d}$, its roots $\overline{r_i}$ all lie in this field. The minimal field of definition of each $\overline{r_i}$ is $\mathbb{F}_{q^{f_i}}$, so the splitting field of $p(x_0, t)$ is $\mathbb{F}_{q^{\mathrm{lcm}(f_1, \ldots, f_r)}} = \mathbb{F}_{q^d}$.
\end{proof}

\begin{rmk}
$\empty$\par
\begin{enumerate}
    \item If $p(x, t)$ is irreducible and separable as a polynomial in $K[t]$, then the set of $x_0$ such that $p(x_0,t)$ has a repeated root forms a closed subset of $\mathbb{A}^1$ under the Zariski topology. Such $x_0$ are the roots of the discriminant of $p(x,t)$ with respect to $t$.
    \item If the prime $(x-x_0)$ is unramified in $L$, it does not imply that $p(x_0, t)$ has distinct roots in $\overline{\mathbb{F}_q}$. For a concrete example, we consider $p(x,t) = t^2-x^3-x^2\in \mathbb{F}_q[x, t]$ with $q>2$. The splitting field of $p$ over $K = \mathbb{F}_q(x)$ is $L = K(\sqrt{x+1})$. The prime $(x)$ splits in $L$ as $(\sqrt{x+1}+1)(\sqrt{x+1}-1)$ yet $p(0, t)$ has repeated roots. We will explain this behavior in the next subsection, and show that if the affine curve defined by $p(x,t) = 0$ is normal (smooth), then $(x-x_0)$ being unramified implies that $p(x_0, t)$ has distinct roots. We will also see that this is related to the conductor condition in the Dedekind-Kummer Theorem.
    \item We also note that a prime ideal in $O_K$ is unramified in $L$ if and only if it is unramified in $K(r_i)$ \cite[Chapter~4; Theorem~31]{marcus1977number}. It is clear that unramified in $L$ implies unramified in the intermediate extension $K(r_i)$. Conversely, $L$ is the Galois closure of $K(r_i)$, constructed by the composite of Galois conjugates of $K(r_i)$. The ramification behavior does not change in a Galois conjugate nor when taking the composite.
\end{enumerate}
\end{rmk}

\subsection*{Geometric Picture and Normalization}
Geometrically, $p(x,t)$ defines an affine curve $C_p = \{p(x,t) = 0\}\subset \mathbb{A}^2$ together with the natural projection $(x,t)\mapsto x$ onto $\mathbb{A}^1$. The fiber over $x_0\in\mathbb{F}_q$ is exactly the roots of the polynomial $p(x_0, t)$. Thus, studying how this family of polynomials splits is equivalent to studying the fiber of this projection.

We first show that if the affine curve $C_p$ is normal (smooth), then $(x-x_0)$ being unramified in $L$ implies that $p(x_0,t)$ has distinct roots.

Let $K(r_i)$ be the field extension of $K$ obtained by adjoining one root of $p(x,t)$. We have $O_K[r_i]\subset K(r_i)$, and $O_K[r_i] = O_K[t]/p(x,t) = \mathbb{F}_q[x,t]/p(x,t)$ is the coordinate ring of the curve $C_p$. Let $\mathcal{O}$ be the ring of integers in $K(r_i)$, which is also the normalization of $O_K[r_i]$ in $K(r_i)$. Since $C_p$ is normal, we have
$\mathcal{O} = O_K[r_i]$, and the conductor $[\mathcal{O}: O_K[r_i]]$ is trivial. By the Dedekind-Kummer theorem, the splitting of a prime $\mathfrak{p} \in \mathrm{Spec}(O_K)$ in $K(r_i)$ is determined by the factorization of $p \bmod \mathfrak{p}$ in $(O_K/\mathfrak{p})[t]$. Since $(x - x_0)$ is unramified in $L$, it is unramified in every subextension, including $K(r_i)$. Thus, the reduction $p(x_0, t)$ must have no repeated factors in $\mathbb{F}_q[t]$, and hence has distinct roots in $\overline{\mathbb{F}_q}$.

On the other hand, if $C_p$ is not smooth, and the fiber over $x_0$ contains singular points of $C_p$, then the situation becomes more
subtle. When we say that $(x-x_0)$ is unramified in $K(r_i)$ (or equivalently unramified in $L$), we really mean that it splits in the ring of integers in $K(r_i)$, where we implicitly take the normalization. In this case, the polynomial $p(x_0,t)$ may have repeated roots, which
reflects the presence of singularities in $C_p$. However, these singularities
are resolved in the normalization $\widetilde{C_p}$, where the fiber may split into
several distinct points. This is why we have an unramified prime $(x-x_0)$, yet $p(x_0,t)$ has repeated roots.

\begin{eg}
    In the example $p(x,t) = t^2-x^3-x^2$ that we discussed above, the point $(0,0)$ is a singular point on $\mathcal{C}_p$ corresponding to a double point,
    but resolved in its normalization $\widetilde{\mathcal{C}_p}=\{t^2 - x -1 = 0\}$ as two distinct points $(0,\pm 1)$. This explains how it is possible for $(x)$ to be unramified, while $p(0, t)$ still has repeated roots.

    Geometrically, we have this picture:
    
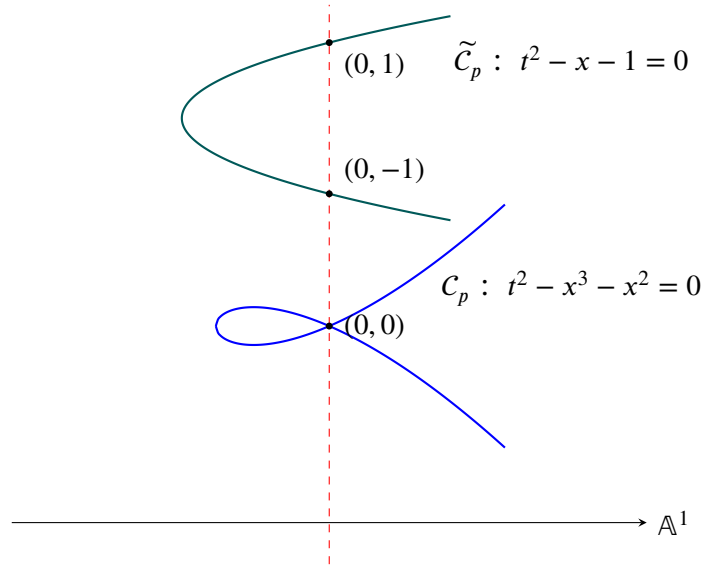
\begin{center}
    \begin{tikzpicture}[>=stealth]
        \draw[dashed,red] (0,-3.5) -- (0,3.9);

        \draw[thick,teal!70!black,smooth,domain=-1.35:1.35,samples=140] plot ({1.95*(\x*\x-1)},{\x+2.4});
        \fill (0,3.4) circle (1.3pt) node[below right] {$\,(0,1)$};
        \fill (0,1.4) circle (1.3pt) node[above right] {$\,(0,-1)$};
        \node at (3.2,3.15) {$\widetilde{\mathcal{C}_p}:\ t^2-x-1=0$};

        \draw[thick,blue,smooth,domain=-1:1.55,samples=140] plot ({1.5*\x},{0.65*\x*sqrt(\x+1)-0.35});
        \draw[thick,blue,smooth,domain=-1:1.55,samples=140] plot ({1.5*\x},{-0.65*\x*sqrt(\x+1)-0.35});
        \fill (0,-0.35) circle (1.3pt) node[right] {$\,(0,0)$};
        \node at (3.2,0.15) {$\mathcal{C}_p:\ t^2-x^3-x^2=0$};

        \draw[->] (-4.2,-2.95) -- (4.2,-2.95) node[right] {$\mathbb{A}^1$};
    \end{tikzpicture}
    \end{center}
\end{eg}

\newpage
\section{Splitting of Polynomials via Galois Theory II}\label{sec:splitting_II}
\subsection*{Obstruction for the Generalization}
We wish to generalize the above result to the following setting:
\begin{itemize}
    \item $X$: an affine normal variety over $\mathbb{F}_q$.
    \item $A = \mathbb{F}_q[X]$: the coordinate ring of $X$.
    \item $p(x, t)\in A[t]$: a degree $n$ irreducible separable polynomial in the variable $t$.
    \item $K$: the function field of the affine variety $X$. It is also the field of fractions of $A$.
    \item $Y$: the affine variety defined by $Y= \{(x,t)\in X\times \mathbb{A}^1 : p(x,t)=0\}$.
    \item $B = A[t]/(p(x,t)) = \mathbb{F}_q[Y]$: the coordinate ring of $Y$.
    \item $K(r_i)$: the finite separable field extension of $K$ by adjoining one root of $p(x,t)$. It is also the function field of the variety $Y$.
    \item $L$: the splitting field of $p(x,t)$ over $K$. It is also the Galois closure of $K(r_i)$.
\end{itemize}

We view $p(x,t)$ as a family of polynomials parametrized by the variety $X$. For a given $\mathbb{F}_q$-rational point $x_0\in X(\mathbb{F}_q)$, we want to study the splitting of the specialized polynomial $p(x_0,t)$. However, different from the previous setting, we have the following obstructions.
\begin{itemize}
    \item There is no notion of the ring of integers in the function field. Instead, we will work with the coordinate ring of $Y$ directly. If normalization is necessary, we take the integral closure of $A$ in $K(Y)$ or $L$. 
    \item There is no unique prime factorization of ideals. Instead, for a maximal ideal $\mathfrak{m}_{x_0}\in A$, we look at the structure of the $\mathbb{F}_q$-algebra $B\otimes_{A} A/\mathfrak{m}_{x_0}\cong B/\mathfrak{m}_{x_0}B$.
    \item There is no notion of unramified at a maximal ideal $\mathfrak{m}_{x_0}$. We will define the notion of \etale (unramified) at a point $x_0$ using the properties about the $\mathbb{F}_q$-algebra $B/\mathfrak{m}_{x_0}B$.
\end{itemize}

We first prove a structure theorem for the $\mathbb{F}_q$-algebra $B/\mathfrak{m}_{x}B$.
\begin{lem}
    In the above setting, the ring $B/\mathfrak{m}_{x}B$ is a finite dimensional $\mathbb{F}_q$-algebra, and its $\mathbb{F}_q$-dimension equals $[K(r_i): K] = n$.
\end{lem}
\begin{proof}
    Since $B = A[t]/(p)$ and $p$ has degree $n$, we know that $B$ is a free $A$-module with basis $\{1,\ldots, t^{n-1}\}$. We get the isomorphism
    $$B\cong A^n.$$
    Since tensor product commutes with the direct sum, we get the isomorphism
    $$B/\mathfrak{m}_x B\cong B\otimes_{A} A/\mathfrak{m}_x\cong A^n\otimes_{A} A/\mathfrak{m}_x\cong (A/\mathfrak{m}_x)^n\cong \mathbb{F}_q^n.$$
    Thus, $B/\mathfrak{m}_xB$ is a finite dimensional $\mathbb{F}_q$-algebra with dimension $n$.
\end{proof}

Now, we define what we meant by $Y\to X$ is \textbf{\etale(unramified)} at a rational point $x\in X(\mathbb{F}_q)$ by using the structure of $B/\mathfrak{m}_xB$ as a $\mathbb{F}_q$-algebra. Recall that a finite-dimensional vector space over $k$
is an Artinian ring, and we have the following
structure theorem for Artinian rings
(see Eisenbud \cite[Theorem~2.14 and Corollary~2.16]{eisenbud1995commutative}).

\begin{thm}[Structure Theorem for Artinian Rings]\label{thm: structure theorem for Artinian ring}
    Let $A$ be an Artinian ring, then $A$ decomposes uniquely (up to isomorphism)
    as a finite direct product of local Artinian rings:
    $$A \cong  A_1\times \cdots\times A_r.$$
\end{thm}
In our case where $A$ is a finite dimensional vector space over $k$,
we also know that for each local Artinian ring $A_i$,
its unique maximal ideal $\mathfrak{m}_i$ is nilpotent,
and the quotient $L_i = A_i/\mathfrak{m}_i$
is a finite field extension of $k$.
\begin{defn}
    We say that a finite dimensional $k$-algebra is \textbf{étale} if it is isomorphic to a finite direct product of finite separable field extensions of $k$.

    In our setting, we say that $Y\to X$ is \textbf{\etale (unramified)} at a rational point $x\in X(\mathbb{F}_q)$ if $B/\mathfrak{m}_x B$ is a finite dimensional \etale algebra over $\mathbb{F}_q$.
\end{defn}

\begin{rmk}
    In modern algebraic geometry, a morphism is \'etale if and only if it is both unramified and flat. In our setting, because $p(x,t)$ is a monic polynomial of degree $n$, the coordinate ring $B = A[t]/(p(x,t))$ is a free $A$-module of rank $n$. This guarantees that the morphism $Y \to X$ is finite and flat. Therefore, verifying that the fiber $B/\mathfrak{m}_x B$ is an \'etale algebra over $\mathbb{F}_q$---which is the standard geometric definition of being \textbf{unramified} at $x$---is equivalent to the morphism being \textbf{\'etale} at $x$. This justifies our interchangeable use of the terminologies.
\end{rmk}

We now formally show that this abstract definition of an étale algebra perfectly recovers the condition that the specialization $p(x_0,t)$ has distinct roots. Furthermore, we establish that the structural decomposition of the étale algebra $B/\mathfrak{m}_{x_0}B$ serves as our replacement for prime ideal factorization, directly encoding the splitting type of $p(x_0, t)$. This is an analogue of the classic Dedekind-Kummer Theorem.

\begin{thm}\label{thm: etale decomposition and splitting}
    Let $x_0 \in X$ be an $\mathbb{F}_q$-rational point with corresponding maximal ideal $\mathfrak{m}_{x_0} \subset A$. Let $B = A[t]/(p(x,t))$ be the coordinate ring of $Y$.
    \begin{enumerate}
        \item Suppose that $p(x_0, t)\in \mathbb{F}_q[t]$ factors as
        $$p(x_0, t) = p_1(t)^{e_1}\cdots p_r(t)^{e_r}.$$
        Then, we have an $\mathbb{F}_q$-algebra isomorphism:
        $$ B/\mathfrak{m}_{x_0}B \cong L_1 \times \cdots \times L_r $$
        where $L_i \cong \mathbb{F}_q[t]/(p_i(t)^{e_i})$. Let $f_i$ be the degree of the field extension $[\mathbb{F}_q[t]/(p_i(t)): \mathbb{F}_q]$; we have
        $$\sum_{i=1}^r e_if_i = \deg p.$$

        \item The morphism $Y \to X$ is étale over $x_0$ if and only if the specialized polynomial $p(x_0, t)$ has distinct roots in $\overline{\mathbb{F}_q}$.
    \end{enumerate}
\end{thm}

\begin{proof}
    By definition, the coordinate ring $B$ equals $A[t]/(p(x,t))$. Using the standard isomorphism for quotient rings, we have
    $$ B/\mathfrak{m}_{x_0}B \cong A[t] / \left( (p(x,t)) + \mathfrak{m}_{x_0}A[t] \right). $$
    The evaluation map at the $\mathbb{F}_q$-rational point $x_0$ defines a surjective homomorphism $A \to \mathbb{F}_q$ with kernel $\mathfrak{m}_{x_0}$. This extends coefficient-wise to a surjective homomorphism $A[t] \to \mathbb{F}_q[t]$ with kernel $\mathfrak{m}_{x_0}A[t]$, yielding the isomorphism
    $$ A[t]/\mathfrak{m}_{x_0}A[t] \cong \mathbb{F}_q[t]. $$
    Under this isomorphism, the image of $p(x,t)$ is exactly the specialized polynomial $p(x_0, t)$. Quotienting both sides by this principal ideal gives the natural isomorphism:
    $$ B/\mathfrak{m}_{x_0}B \cong A[t] / \left( (p(x,t)) + \mathfrak{m}_{x_0}A[t] \right)\cong \mathbb{F}_q[t]/(p(x_0, t)). $$

    Since $\mathbb{F}_q[t]$ is a Principal Ideal Domain, we can uniquely factor $p(x_0, t)$ into irreducible monic polynomials as
    $$ p(x_0, t) = p_1(t)^{e_1} \cdots p_r(t)^{e_r}.$$
    By the Chinese Remainder Theorem, the $\mathbb{F}_q$-algebra $\mathbb{F}_q[t]/(p(x_0, t))$ decomposes as
    $$  B/\mathfrak{m}_{x_0}B\cong \mathbb{F}_q[t]/(p(x_0, t)) \cong \prod_{i=1}^r \mathbb{F}_q[t]/(p_i(t)^{e_i}) .$$
    Since $f_i = [\mathbb{F}_q[t]/(p_i(t)): \mathbb{F}_q]$ is the degree of the polynomial $p_i(t)$, by comparing degrees in the factorization $p(x_0, t) = p_1(t)^{e_1} \cdots p_r(t)^{e_r}$, we see that $\sum_{i=1}^r e_if_i = \deg(p) = n$, which establishes (1).

    For this finite-dimensional algebra to be étale, it must be isomorphic to a direct product of fields. The local ring $\mathbb{F}_q[t]/(p_i(t)^{e_i})$ is a field if and only if $e_i = 1$. If $e_i > 1$, the ring contains non-zero nilpotent elements and cannot be a field. Therefore, $B/\mathfrak{m}_{x_0}B$ is an étale algebra over $\mathbb{F}_q$ if and only if $e_i = 1$ for all $i$. This is exactly the condition that $p(x_0, t)$ has $n$ distinct roots in $\overline{\mathbb{F}_q}$, proving (2).
\end{proof}

\subsection{Galois Action and the Frobenius Element}

To understand the statistical distribution of the splitting types of $p(x,t)$ over the variety $X$, we must bridge the global Galois action over the function field $K$ with the local Galois action over the finite field $\mathbb{F}_q$. Let $x_0\in X$ be a $\mathbb{F}_q$-rational point where $Y\to X$ is \etale. We will define $\Frob_{x_0}$, the Frobenius at the point $x_0$.

Let $K$ be the function field of $X$, and let $p(x,t) \in A[t]$ be our defining monic irreducible separable polynomial. Let $L$ be the splitting field of $p(x,t)$ over $K$, and let $r_1, \dots, r_n \in L$ be its distinct roots. The finite Galois group $G = \Gal(L/K)$ acts faithfully on the set of roots $\{r_1, \dots, r_n\}$, which induces an embedding of $G$ into the symmetric group $S_n$. We will view elements in $G$ as elements in $S_n$, and speak of their cycle types.

To connect this to the rational points of $X$, let $\mathcal{O}$ be the integral closure of $A$ in $L$. For a fixed $\mathbb{F}_q$-rational point $x_0 \in X$ corresponding to the maximal ideal $\mathfrak{m}_{x_0} \subset A$, the lying over theorem for integral extensions guarantees the existence of a maximal ideal $\mathfrak{M} \subset \mathcal{O}$ lying above $\mathfrak{m}_{x_0}$ (\cite[Theorem~5.10, Corollary~5.8]{atiyah1969introduction}). We define the \textbf{decomposition group} $D_{\mathfrak{M}}$ and the \textbf{inertia group} $I_{\mathfrak{M}}$ as:
$$ D_{\mathfrak{M}} = \{\sigma \in \Gal(L/K) \mid \sigma(\mathfrak{M}) = \mathfrak{M}\} $$
$$ I_{\mathfrak{M}} = \{\sigma \in D_{\mathfrak{M}} \mid \sigma(a) \equiv a \pmod{\mathfrak{M}} \text{ for all } a \in \mathcal{O} \} $$

Any $\sigma \in D_{\mathfrak{M}}$ induces a well-defined automorphism on the finite residue field $\mathcal{O}/\mathfrak{M}$. This gives a natural surjective homomorphism $D_{\mathfrak{M}} \to \Gal((\mathcal{O}/\mathfrak{M}) / \mathbb{F}_q)$ whose kernel is exactly the inertia group $I_{\mathfrak{M}}$. We first prove a lemma relating the étaleness with the inertia group defined here.

\begin{lem}\label{lem:unramified_splitting_field}
Under the above settings, if the reduced polynomial $\overline{p}(t) = p(t) \pmod{\mathfrak{m}}$ has distinct roots in an algebraic closure of $A/\mathfrak{m}$, then for any maximal ideal $\mathfrak{M} \subset \mathcal{O}$ lying over $\mathfrak{m}$, the inertia group $I_{\mathfrak{M}}$ is trivial.
\end{lem}

\begin{proof}
    Let $r_1, \dots, r_n$ be the distinct roots of $p(t)$ in $L$. Since $p(t)$ has coefficients in $A$, $r_i$ are integral over $A$, and we have $r_1, \dots, r_n \in \mathcal{O}$. We have the factorization $p(t) = \prod_{i=1}^n (t - r_i)$ in $\mathcal{O}[t]$. Reducing this equation modulo $\mathfrak{M}$, we obtain:
    $$ \overline{p}(t) = \prod_{i=1}^n (t - \overline{r_i}) \in (\mathcal{O}/\mathfrak{M})[t] $$
    where $\overline{r_i} = r_i \pmod{\mathfrak{M}}$. By assumption, the polynomial $\overline{p}(t)$ has completely distinct roots. This implies that the reduced roots $\overline{r_1}, \dots, \overline{r_n}$ are all distinct elements in the residue field $\mathcal{O}/\mathfrak{M}$.

    Now, let $\sigma \in I_{\mathfrak{M}}$ be any element of the inertia group. By definition, $\sigma$ is a Galois automorphism in $\Gal(L/K)$ such that $\sigma(a) \equiv a \pmod{\mathfrak{M}}$ for all $a \in \mathcal{O}$. Because $\sigma$ fixes $K$, it must permute the roots of $p(t)$. Thus, for any root $r_i$, we have $\sigma(r_i) = r_j$ for some index $j$. However, applying the defining property of the inertia group to the integral element $r_i$ yields:
    $$ r_j = \sigma(r_i) \equiv r_i \pmod{\mathfrak{M}} $$
    This congruence means that $\overline{r_j} = \overline{r_i}$ in $\mathcal{O}/\mathfrak{M}$. Since we established that the reduced roots are distinct, we must have $j = i$.

    Therefore, $\sigma(r_i) = r_i$ for every $i \in \{1, \dots, n\}$. Because the splitting field is generated by these roots, $L = K(r_1, \dots, r_n)$, the automorphism $\sigma$ must be the identity on $L$. This proves that $I_{\mathfrak{M}} = \{\mathrm{id}\}$.
\end{proof}

Suppose that the morphism $Y \to X$ is étale over $x_0$, the specialized polynomial $p(x_0, t)$ has distinct roots in $\overline{\mathbb{F}_q}$. By Lemma \ref{lem:unramified_splitting_field}, this implies that the inertia group $I_{\mathfrak{M}}$ is trivial. Consequently, we obtain a canonical isomorphism:
$$ D_{\mathfrak{M}} \cong \Gal((\mathcal{O}/\mathfrak{M}) / \mathbb{F}_q) $$

The Galois group of the finite residue field is generated by the Frobenius automorphism $\sigma_q(a) = a^q$. Through the isomorphism above, $\sigma_q$ lifts to a unique element in $D_{\mathfrak{M}} \subset \Gal(L/K)$, called the \textbf{Frobenius element} at $\mathfrak{M}$, denoted $\Frob_{\mathfrak{M}}$. If we choose a different maximal ideal $\mathfrak{M}'$ over $\mathfrak{m}_{x_0}$, the resulting Frobenius element is conjugate to $\Frob_{\mathfrak{M}}$ in $G$. Thus, the conjugacy class of this element depends only on $x_0$, and we denote this conjugacy class by $\Frob_{x_0}$.

We can now state the higher-dimensional analogue of Theorem \ref{thm:spliting of polynomials via Galois theory I}.

\begin{thm}\label{thm:spliting of polynomials via Galois theory II}
Let $x_0 \in X$ be a $\mathbb{F}_q$-rational point such that the morphism $Y \to X$ is étale over $x_0$. Suppose the specialized polynomial $p(x_0, t)$ has the unique factorization $p_1(t) \cdots p_r(t)$ in $\mathbb{F}_q[t]$, with degrees $f_i = \deg(p_i)$. Then:
\begin{enumerate}
    \item The roots of $p(x,t)$ modulo $\mathfrak{M}$ are exactly the roots of $p(x_0, t)$, and the reduction map $\rho_i \mapsto \rho_i \bmod \mathfrak{M}$ is a bijection from the global roots in $L$ to the local roots in $\mathcal{O}/\mathfrak{M}$.
    \item The cycle type of the Frobenius element $\Frob_{x_0} \in \Gal(L/K)$ acting on the global roots $\{\rho_1, \dots, \rho_n\}$ coincides perfectly with the splitting type $(f_1, \dots, f_r)$ of the specialized polynomial $p(x_0,t)$.
    \item Let $d = \mathrm{lcm}(f_1, \dots, f_r)$. The splitting field of $p(x_0, t)$ is exactly $\mathbb{F}_{q^d}$.
    \item The $\mathbb{F}_q$-algebra $B/\mathfrak{m}_{x_0}B$ serving as the fiber of $Y \to X$ over $x_0$ decomposes as:
    $$ B/\mathfrak{m}_{x_0}B \cong L_1 \times \dots \times L_r $$
    where $[L_i : \mathbb{F}_q] = f_i$.
\end{enumerate}
\end{thm}

\begin{proof}
    Because the roots $r_i$ of the monic polynomial $p(x,t) \in A[t]$ are integral over $A$, they lie in $\mathcal{O}$. Reducing the factorization $p(x,t) = \prod_{i=1}^n (t - r_i)$ modulo $\mathfrak{M}$ yields:
    $$ p(x_0, t) = \prod_{i=1}^n (t - \overline{r_i}) \in (\mathcal{O}/\mathfrak{M})[t] $$
    where $\overline{r_i} = r_i \bmod \mathfrak{M}$. Because $Y \to X$ is étale over $x_0$, $p(x_0, t)$ has exactly $n$ distinct roots in $\overline{\mathbb{F}_q}$. Thus, the $n$ reduced elements $\overline{r_i}$ must be distinct, proving that the reduction map is a bijection, which establishes (1).

    For (2), the definition of $\Frob_{\mathfrak{M}}$ ensures that for any $a \in \mathcal{O}$, $\Frob_{\mathfrak{M}}(a) \equiv a^q \pmod{\mathfrak{M}}$. Applying this to the global roots gives $\Frob_{\mathfrak{M}}(r_i) \equiv \overline{r_i}^q \pmod{\mathfrak{M}}$. Therefore, the permutation action of $\Frob_{\mathfrak{M}}$ on the global roots $\{r_1, \dots, r_n\}$ exactly mirrors the permutation action of the Frobenius $\sigma_q$ on the local roots $\{\overline{r_1}, \dots, \overline{r_n}\}$. Since the roots of an irreducible polynomial $p_i(t)$ of degree $f_i$ form a single cycle of length $f_i$ under $\sigma_q$, the cycle type of $\Frob_{x_0}$ is exactly the splitting type $(f_1, \dots, f_r)$.

    For (3), the decomposition group $D_{\mathfrak{M}}$ is cyclically generated by $\Frob_{\mathfrak{M}}$. The order of $\Frob_{\mathfrak{M}}$ in $S_n$ is the least common multiple of its cycle lengths, so $|D_{\mathfrak{M}}| = d = \mathrm{lcm}(f_1, \dots, f_r)$. Because $I_{\mathfrak{M}}$ is trivial, the Galois group of the residue field extension has order $d$, meaning the inertia degree is $d$ and $\mathcal{O}/\mathfrak{M} \cong \mathbb{F}_{q^d}$.

    Finally, (4) is a direct consequence of Theorem \ref{thm: etale decomposition and splitting}, confirming that the étale fiber structure $B/\mathfrak{m}_{x_0}B$ exactly encodes the cycle lengths of the Frobenius element.
\end{proof}

\subsection{Chebotarev's Density Theorem and Splitting Statistics}

Because the polynomial $p(x,t)$ defines a finite separable extension, the discriminant $\Delta(p)$ is a non-zero element of $A$. The locus of points where the morphism $Y \to X$ fails to be étale is exactly the closed subvariety $V(\Delta(p)) \subset X$.

Let $U = X \setminus V(\Delta(p))$ be the dense open subvariety consisting of all points where $Y \to X$ is étale. For any $\mathbb{F}_q$-rational point $x_0 \in U(\mathbb{F}_q)$, Theorem \ref{thm:spliting of polynomials via Galois theory II} shows that the Frobenius element $\Frob_{x_0}$ is a well-defined conjugacy class in the Galois group $G = \Gal(L/K)$.

To determine the statistical distribution of these conjugacy classes, we must account for the geometric and arithmetic parts of the Galois group. Let $\mathbb{F}_{q^m} = L \cap \overline{\mathbb{F}_q}$ be the constant field of the splitting field $L$. We can visualize the relationship between these subfields with the following diagram:

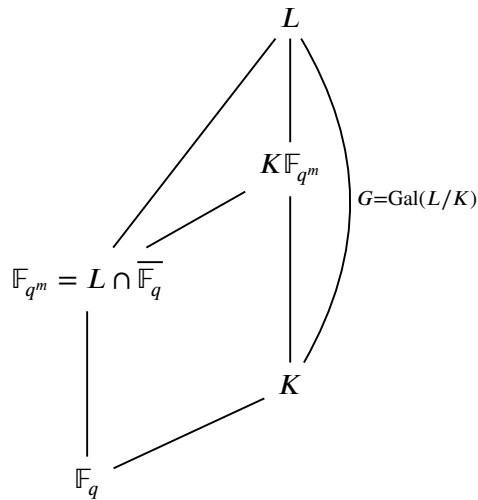
\begin{center}
\begin{tikzcd}
& L \arrow[dddd, "G = \operatorname{Gal}(L/K)", no head, bend left]\\& \\
& K\mathbb{F}_{q^m} \arrow[uu, no head]     \\
\mathbb{F}_{q^m}=L\cap \overline{\mathbb{F}_q} \arrow[ru, no head] \arrow[ruuu, no head] &   \\ & K \arrow[uu, no head]     \\
\mathbb{F}_q \arrow[uu, no head] \arrow[ru, no head] &
\end{tikzcd}
\end{center}

This restriction induces a natural homomorphism:
$$ \pi: G \to \Gal(\mathbb{F}_{q^m} / \mathbb{F}_q) $$
defined by $\pi(\sigma) = \sigma|_{\mathbb{F}_{q^m}}$. Because $X$ is a variety over $\mathbb{F}_q$, we have $K \cap \overline{\mathbb{F}_q} = \mathbb{F}_q$. Consequently, $K \cap \mathbb{F}_{q^m} = \mathbb{F}_q$. Then, the Galois group of the composite extension $K\mathbb{F}_{q^m} / K$ is canonically isomorphic to $\Gal(\mathbb{F}_{q^m}/\mathbb{F}_q)$. Since $K\mathbb{F}_{q^m}$ is an intermediate subfield of $L/K$, the fundamental theorem of Galois theory ensures that the restriction map onto this quotient is surjective.

The kernel of $\pi$ is the \textbf{geometric Galois group} $G^{\mathrm{geom}} = \Gal(L / (K \mathbb{F}_{q^m}))$, and $G$ is called the \textbf{arithmetic Galois group}. We have the \textbf{fundamental exact sequence}
$$1\to G^{\mathrm{geom}}\to G\xrightarrow{\pi} \Gal(\mathbb{F}_{q^m}/\mathbb{F}_q)\to 1.$$

For any rational point $x_0 \in U(\mathbb{F}_q)$, recall that the Frobenius element $\Frob_{x_0}$ is defined by the congruence $\Frob_{x_0}(a) \equiv a^q \pmod{\mathfrak{M}}$ for all $a \in \mathcal{O}$. To see how this acts on the constant field $\mathbb{F}_{q^m}$, we first verify that $\mathbb{F}_{q^m}$ is contained in $\mathcal{O}$. Because $A$ is an $\mathbb{F}_q$-algebra, we have $\mathbb{F}_q \subset A$. Any element $\alpha \in \mathbb{F}_{q^m}$ is algebraic over $\mathbb{F}_q$, meaning it is a root of a monic polynomial in $\mathbb{F}_q[t] \subset A[t]$. Thus, $\alpha$ is integral over $A$, and $\alpha\in \mathcal{O}$. This implies the inclusion $\mathbb{F}_{q^m} \subset \mathcal{O}$.

Now, applying the Frobenius congruence to any element $\alpha \in \mathbb{F}_{q^m}$, we have $\Frob_{x_0}(\alpha) - \alpha^q \in \mathfrak{M}$. However, both $\Frob_{x_0}(\alpha)$ and $\alpha^q$ belong to the constant field $\mathbb{F}_{q^m}$. This means their difference must lie in the intersection $\mathfrak{M} \cap \mathbb{F}_{q^m}$. Because $\mathbb{F}_{q^m}$ is a field and $\mathfrak{M}$ is a proper ideal of $\mathcal{O}$, their intersection contains no invertible elements and must therefore be trivial. Consequently, $\Frob_{x_0}(\alpha) - \alpha^q = 0$, meaning $\Frob_{x_0}$ acts on $\mathbb{F}_{q^m}$ precisely as the standard Frobenius automorphism $\sigma_q(\alpha) = \alpha^q$. Because $\pi$ is defined as restriction to $\mathbb{F}_{q^m}$, it follows that $\Frob_{x_0}$ always lies in the specific coset of $G$ defined by $\pi^{-1}(\sigma_q)$.

With this restriction in mind, the asymptotic distribution of the Frobenius elements is governed by the Chebotarev Density Theorem. For the one-dimensional case, we refer to Jarden's elementary approach \cite[Lemma~1]{Jarden1982}, and for the higher-dimensional case over finite fields, we refer to Meagher \cite{MEAGHER_2018} and Kowalski \cite[Chapter~8]{kowalski2008large}, as well as the comprehensive treatment in Fried and Jarden's \textit{Field Arithmetic} \cite{fried2008field}. We state the following theorem as a corollary of the version by Meagher. Meagher provides an elementary proof using Lang-Weil estimates under the assumption that the extension is regular (i.e., $m=1$ and $G = G^{\mathrm{geom}}$). For a general extension where $m \ge 1$, we can deduce the following result by applying Meagher's theorem to the base-changed variety $X \times_{\mathbb{F}_q} \mathbb{F}_{q^m}$ over $\mathbb{F}_{q^m}$.

\begin{thm}[Chebotarev's Density Theorem over Finite Fields]\label{thm:chebotarev}
    Let $X$ be an affine normal variety of dimension $d$ over $\mathbb{F}_q$, and let $C \subset G$ be a conjugacy class. Let $S_C(q)$ be the number of rational points $x_0 \in U(\mathbb{F}_q)$ such that $\Frob_{x_0} = C$. Let $m = [\mathbb{F}_{q^m} : \mathbb{F}_q]$ be the degree of the constant field extension, and let $G^{\mathrm{geom}}$ be the geometric Galois group.
    \begin{enumerate}
        \item If $\pi(C) \neq \sigma_q$, then $S_C(q) = 0$.
        \item If $\pi(C) = \sigma_q$, then:
        $$ S_C(q) = \frac{|C|}{|G^{\mathrm{geom}}|} q^d + O(q^{d - 1/2}) = m \frac{|C|}{|G|} q^d + O(q^{d - 1/2}) $$
        where the implied constant is independent of $q$.
    \end{enumerate}
\end{thm}

Since the total number of rational points on $U$ is asymptotically $q^d$, the theorem implies that if $C$ restricts to $\sigma_q$, the proportion of points yielding the conjugacy class $C$ tends to $\frac{|C|}{|G^{\mathrm{geom}}|}$. With Theorem \ref{thm:spliting of polynomials via Galois theory II}, this provides a complete mechanism for determining the asymptotic probability of any splitting type.

\begin{cor}\label{cor:splitting_density}
    Let $X$ be an affine normal variety of dimension $d$ over $\mathbb{F}_q$ with coordinate ring $\mathbb{F}_q[X]$, and let $p(x, t)\in \mathbb{F}_q[X][t]$ be a monic polynomial that is separable over the function field $K = \mathbb{F}_q(X)$. Let $L$ be the splitting field of $p$ over $K$, with Galois group $G = \Gal(L/K) \subset S_n$, and geometric Galois group $G^{\mathrm{geom}}$.

    Let $\mathcal{P}(f_1, \dots, f_r)$ be the asymptotic proportion of rational points $x_0 \in X(\mathbb{F}_q)$ for which the specialized polynomial $p(x_0, t)$ has the splitting type $(f_1, \dots, f_r)$.

    Let $C(f_1, \dots, f_r) \subset G$ be the set of elements in the Galois group whose cycle type is $(f_1, \dots, f_r)$. If $\pi(C(f_1, \dots, f_r)) = \sigma_q$, then the proportion is given by:
    $$ \mathcal{P}(f_1, \dots, f_r) = \frac{|C(f_1, \dots, f_r)|}{|G^{\mathrm{geom}}|} +O(q^{-1/2}) = m \frac{|C(f_1, \dots, f_r)|}{|G|}+O(q^{-1/2}) $$
    Otherwise, if the restriction does not match $\sigma_q$, the proportion is $0$.
\end{cor}

\begin{proof}
    Let $U \subset X$ be the dense open subvariety over which the specialization $p(x_0,t)$ has distinct roots in $\overline{\mathbb{F}_q}$. By Theorem \ref{thm:spliting of polynomials via Galois theory II}, for any $x_0 \in U(\mathbb{F}_q)$, the splitting type of the specialized polynomial $p(x_0, t)$ is determined by the cycle type of its Frobenius element $\Frob_{x_0}$.

    Applying Theorem \ref{thm:chebotarev}, the number of points in $U(\mathbb{F}_q)$ yielding the conjugacy class $C = C(f_1, \dots, f_r)$ is:
    $$ S_C(q) = m \frac{|C|}{|G|} q^d + O(q^{d - 1/2}). $$
    The points we excluded lie in the complement $X \setminus U$, which is a closed subvariety of dimension at most $d-1$. By the Lang-Weil estimates, the number of rational points in this exceptional set is bounded by $O(q^{d-1})$.

    Since the total number of rational points on $X$ is $|X(\mathbb{F}_q)| = q^d + O(q^{d-1/2})$, extending our count from $U(\mathbb{F}_q)$ to all of $X(\mathbb{F}_q)$ introduces an error of at most $O(q^{d-1})$. Because this is strictly smaller than the Chebotarev error term $O(q^{d-1/2})$, it is completely absorbed into the asymptotic bound. Dividing the expected number of points $S_C(q)$ by $|X(\mathbb{F}_q)|$ yields the stated asymptotic density.
\end{proof}

\begin{rmk}\label{rmk: can drop affine and normal condtion}
    Although we require $X$ to be an affine normal and geometrically integral variety over $\mathbb{F}_q$ in order to simplify the proof, we can actually drop the affine and normal conditions, and the above Corollary \ref{cor:splitting_density} still holds. The normality can be dropped since an variety $X$ is generically normal, in the sense that there is an open dense subset of $U\subset X$ that is normal (See e.g \cite[Proposition 2.2.4]{Liu2002} or \cite[Theorem I.5.3]{Hartshorne} for a stronger statement that a variety is generically regular). The affineness can also be dropped, as Theorem \ref{thm: etale decomposition and splitting} also holds for a general variety. We can replace $X$ by either an affine copy containing the point $x_0$, or simply work with the local ring $\mathcal{O}_{X, x_0}$. We also have a well-defined notion of fiber, étaleness and Frobenius in this setup. To apply Meagher's version of Chebotarev density theorem (\cite{MEAGHER_2018}) which requires $X$ to be quasi-projective, we can restrict to an open dense subset of $U\subset X$ that is quasi-projective. Its existence is guaranteed by Chow's lemma (See \cite[Exercise II.4.10]{Hartshorne}). Similar to restricting to the open dense subset $U$ where $f(x_0,t)$ has distinct roots, changing $X$ to an open dense subset does not influence the above result, as the error will be absorbed in $O(q^{-1/2})$
\end{rmk}

\subsection{Application: Kummer Covers and $n$-th Power Characters}

As a first concrete application of Theorem \ref{thm:chebotarev}, we consider the splitting behavior of a single Kummer extension over a variety. We refer to \cite[Chapter~VI, Theorem~8.2]{lang2002algebra} for Kummer Theory.

Let $X$ be an affine normal variety of dimension $d$ over $\mathbb{F}_q$, and let $f \in \mathbb{F}_q[X]$ be a regular function on $X$. Denote by $K = \mathbb{F}_q(X)$ the function field of $X$. We fix an integer $n > 1$ coprime to $q$, and this ensures that $x^n-1$ has $n$ distinct roots in $\overline{K}$, which is required for Kummer Theory. We further assume that $f$ is not a perfect $d$-th power in $\overline{\mathbb{F}_q}(X)$ for any $d \mid n$. We study the factorization of the specialized polynomials $t^n - f(x_0)$ over $\mathbb{F}_q$ where $x_0\in X$.

Let $L = K(\mu_n, \sqrt[n]f)$ be the splitting field of $t^n - f(x)$ over $K$. The constant field of $L$ is $L\cap \overline{\mathbb{F}_q} = \mathbb{F}_q(\mu_n)$, whose degree over $\mathbb{F}_q$ is the multiplicative order of $q$ modulo $n$. We denote this degree of extension by $m$ following the previous convention.

By Kummer Theory, the geometric Galois group over $K\overline{\mathbb{F}_q} = K(\mu_n)$ is $G^{\mathrm{geom}} \cong \mathbb{Z}/n\mathbb{Z}$, generated by the homomorphism sending $\sqrt[n]{f}\to \zeta\sqrt[n]{f}$, where $\zeta\in \mu_n$ is a primitive $n$-th root of unity. The restriction to the constant field dictates that the valid Frobenius elements must map $\zeta \mapsto \zeta^q$.

We pick $x_0\in X$ so that $f(x_0)\neq 0$, which ensures that $t^n-f(x_0)$ has $n$ distinct roots in $\overline{\mathbb{F}_q}$. To determine the splitting type of $t^n - f(x_0)$ over $\mathbb{F}_q$, we must explicitly trace the action of the Frobenius element $\Frob_{x_0}$ on the roots of the polynomial. Let $\alpha_0 = \sqrt[n]{f(x_0)}$ be a fixed root. The $n$ distinct roots in $\overline{\mathbb{F}_q}$ can be naturally indexed by $k \in \mathbb{Z}/n\mathbb{Z}$ as:
$$ \alpha_k = \zeta^k \alpha_0 $$
By the restriction to the constant field, we know the Frobenius element acts on the roots of unity via $\Frob_{x_0}(\zeta) = \zeta^q$. Furthermore, because $\Frob_{x_0}$ must map the root $\alpha_0$ to another root of the same polynomial, there exists a unique index $c \in \mathbb{Z}/n\mathbb{Z}$ such that:
$$ \Frob_{x_0}(\alpha_0) = \zeta^c \alpha_0 $$
Using these two properties, we can completely determine the action of $\Frob_{x_0}$ on any arbitrary root $\alpha_k$:
$$ \Frob_{x_0}(\alpha_k) = \Frob_{x_0}(\zeta^k \alpha_0) = \Frob_{x_0}(\zeta)^k \Frob_{x_0}(\alpha_0) = (\zeta^q)^k (\zeta^c \alpha_0) = \zeta^{qk+c} \alpha_0 = \alpha_{qk+c} $$
Therefore, the action of the Frobenius element on the roots corresponds exactly to the affine permutation of the indices $\mathbb{Z}/n\mathbb{Z}$ given by:
$$ \sigma_{q,c}(k) \equiv qk + c \pmod n $$
The lengths of the disjoint cycles of this permutation $\sigma_{q,c}$ precisely correspond to the degrees of the irreducible factors of $t^n - f(x_0)$ over $\mathbb{F}_q$.

Because the geometric Galois group $G^{\mathrm{geom}} \cong \mathbb{Z}/n\mathbb{Z}$ has size $n$, each possible shift parameter $c \in \mathbb{Z}/n\mathbb{Z}$ corresponds to exactly one element in the geometric Galois group. Applying Corollary \ref{cor:splitting_density}, the parameter $c$ is uniformly distributed. That is, for any fixed $c \in \mathbb{Z}/n\mathbb{Z}$, the density of rational points $x_0 \in X(\mathbb{F}_q)$ yielding that specific shift is exactly $\frac{1}{n}+O(q^{-1/2})$.

We summarize this equidistribution of splitting types in the following theorem.

\begin{thm}[Splitting of Kummer Families]\label{thm:kummer_splitting}
Let $X$ be an affine normal variety over $\mathbb{F}_q$, and let $f \in \mathbb{F}_q[X]$. Assume $n > 1$ is coprime to $q$, and $f$ is not a perfect $d$-th power in $\overline{\mathbb{F}_q}(X)$ for any $d \mid n$. Let $U \subset X$ be the open subvariety where $f(x) \neq 0$.

For each shift parameter $c \in \mathbb{Z}/n\mathbb{Z}$, let $S_c \subset U(\mathbb{F}_q)$ be the set of rational points $x_0$ such that the splitting type of $t^n - f(x_0)$ over $\mathbb{F}_q$ corresponds exactly to the cycle type of the affine permutation on $\mathbb{Z}/n\mathbb{Z}$ given by
$$ \sigma_{q,c}(k) \equiv qk + c \pmod n .$$
Then, as $q \to \infty$, the rational points of $U$ are uniformly distributed among these $n$ possible splitting behaviors. Specifically, for any fixed $c \in \mathbb{Z}/n\mathbb{Z}$, the number of points is:
$$ |S_c|/|U(\mathbb{F}_q)| = \frac{1}{n}  + O(q^{- 1/2}) $$
\end{thm}

With this general theorem established, the exact factorization statistics of $t^n - f(x_0)$ simply reduce to analyzing the combinatorial cycle structure of the map $k \mapsto qk + c \pmod n$. This structure depends heavily on $m$, the multiplicative order of $q$ modulo $n$.

\begin{eg}
    We consider a simple example where $X$ is the affine line over $\mathbb{F}_q$, and $f(x) = x$. We fix $q$ coprime to $3$, and study the density of $x_0\in\mathbb{F}_q$ so that $t^3-x_0$ has a particular splitting type.

    \textbf{Case 1: }If $q\equiv 1 \mod 3$, then $m=1$, and the affine permutation on $\mathbb{Z}/3\mathbb{Z}$ is
    $$k\mapsto k+c \pmod 3.$$
    Picking $c = 0,1,2$ gives the permutation of cycle type $(1,1,1), (3), (3)$ respectively. Therefore, there is around $1/3$ of $x_0\in\mathbb{F}_q$ such that $t^3-x_0$ has $3$ roots over $\mathbb{F}_q$, and around $2/3$ of $x_0\in\mathbb{F}_q$ such that $t^3-x_0$ has $3$ $\mathbb{F}_{q^3}$-conjugate roots.

    \textbf{Case 2: }If $q\equiv -1 \mod 3$, then $m=2$, and the affine
    permutation on $\mathbb{Z}/3\mathbb{Z}$ is
    $$k\mapsto -k +c \pmod 3.$$
    Picking $c = 0,1,2$ gives the permutation of cycle type $(1,2), (1,2), (1,2)$ respectively. Therefore, for almost all $x_0\in\mathbb{F}_q$ (actually, all except $x_0=0$), the polynomial $t^3-x_0$ has one $\mathbb{F}_q$-root and two conjugate $\mathbb{F}_{q^2}$-roots.
\end{eg}

\subsubsection*{The Totally Split Case: $m = 1$ and Connection with the $n$-th Power Characters}
When $q \equiv 1 \pmod n$, the constant field extension is trivial ($m=1$) and $\mathbb{F}_q$ already contains the full set of $n$-th roots of unity, $\mu_n$. The affine map simplifies to a pure translation:
$$ \sigma_{1,c}(k) \equiv k + c \pmod n. $$
For a fixed $c$, every disjoint cycle of this permutation has the exact same length, given by $r = \frac{n}{\gcd(c, n)}$. This implies that $t^n - f(x_0)$ has splitting type $(r,r,\ldots, r)$, and it factors into  $\gcd(c, n)$ number of irreducible polynomials, all of which share the same uniform degree $r$.

Because $c$ is uniformly distributed over $\mathbb{Z}/n\mathbb{Z}$ by Theorem \ref{thm:kummer_splitting}, the probability of observing a specific uniform splitting type is proportional to the number of elements $c$ sharing the same greatest common divisor with $n$. Specifically, for any divisor $r \mid n$, the density of rational points where $t^n - f(x_0)$ splits into $n/r$ irreducible polynomials of degree $r$ is given by:
$$ \mathcal{P}(\text{uniform degree } r) = \frac{\phi(r)}{n} +O(q^{-1/2}),$$
where $\phi$ is Euler's totient function.

Furthermore, in this completely split setting, the translation parameter $c$ is explicitly captured by the \textbf{$n$-th power residue symbol}. Let $\chi_n : \mathbb{F}_q^\times \to \mu_n$ be the $n$-th power character defined by:
$$ \chi_n(a) \equiv a^{\frac{q-1}{n}} \pmod{\mathfrak{p}} $$
For a rational point $x_0 \in U(\mathbb{F}_q)$, the Frobenius element $\Frob_{x_0}$ acts on a fixed root $\alpha_0 = \sqrt[n]{f(x_0)}$ by raising it to the $q$-th power:
$$ \Frob_{x_0}(\alpha_0) \equiv \alpha_0^q = \alpha_0^{q-1} \alpha_0 = f(x_0)^{\frac{q-1}{n}} \alpha_0 = \chi_n(f(x_0)) \alpha_0 $$
Recall that we defined $c$ via the relation $\Frob_{x_0}(\alpha_0) = \zeta^c \alpha_0$ for a fixed primitive root $\zeta \in \mu_n$. Therefore, we have the direct algebraic identification:
$$ \chi_n(f(x_0)) = \zeta^c$$
The cycle type of the polynomial is completely determined by the order of the character value $\chi_n(f(x_0))$ in $\mu_n$. Translating Theorem \ref{thm:kummer_splitting} into the language of characters, we obtain the equidistribution of the $n$-th power residue symbol over the variety $X$. When $X$ is one-dimensional, this result is an application of the Weil's character sum bound. The following theorem should also be viewed as an application of Deligne's work.

\begin{thm}[Equidistribution of $n$-th Power Characters]\label{thm:nth_power_equidistribution}
Let $X$ be an affine normal variety of dimension $d$ over $\mathbb{F}_q$, and $f \in \mathbb{F}_q[X]$ satisfying the conditions of Theorem \ref{thm:kummer_splitting}. Assume further that $q \equiv 1 \pmod n$, and let $\chi_n$ be the $n$-th power character.

Then, the values of $\chi_n(f(x_0))$ are asymptotically uniformly distributed. Specifically, for any fixed root of unity $\zeta^c \in \mu_n$, the proportion of rational points evaluating to $\zeta^c$ is:
$$ \frac{|\{ x_0 \in U(\mathbb{F}_q) \mid \chi_n(f(x_0)) = \zeta^c \}|}{|U(\mathbb{F}_q)|} = \frac{1}{n} + O(q^{-1/2}). $$
\end{thm}

\subsubsection*{The General Case}

When $m > 1$, the affine permutation map on $\mathbb{Z}/n\mathbb{Z}$ includes a non-trivial multiplication by $q$. In this case, $t^n - f(x_0)$ does not necessarily factor into polynomials of equal degrees. We will assume that $m \geq 1$ in the rest of the subsection.

For an arbitrary shift $c$, iterating the affine map $r$ times gives the permutation
$$\sigma_{q,c}^r(k) \equiv q^rk + c\frac{q^r-1}{q-1}\pmod n.$$
The length of the cycle containing the element $k$ is determined by the smallest integer $r \ge 1$ satisfying $q^rk + c\frac{q^r-1}{q-1} \equiv k \pmod n$.

Because $m$ is defined as the multiplicative order of $q$ modulo $n$, we have $q^m \equiv 1 \pmod n$. Iterating the map exactly $m$ times yields:
$$\sigma_{q,c}^m(k) \equiv k + c \frac{q^m - 1}{q - 1} \pmod n.$$
This demonstrates that $\sigma_{q,c}^m$ acts as a pure translation on $\mathbb{Z}/n\mathbb{Z}$ by the constant $T_c \equiv c \frac{q^m - 1}{q - 1} \pmod n$. More precisely, $\sigma_{q,c}^m$ is a product of disjoint cycles of equal length $d = \frac{n}{\gcd(T_c, n)}$. This constrains the possible cycle lengths of the original permutation $\sigma_{q,c}$, and consequently, the degrees of the irreducible factors of $t^n - f(x_0)$.

Suppose $k$ belongs to a cycle of length $r$ under $\sigma_{q,c}$. We can evaluate the map after $d \cdot m$ steps by taking the pure translation $\sigma_{q,c}^m$ and applying it $d$ times. Since each application shifts the input by $T_c$, applying it $d$ times produces a cumulative shift of $d \cdot T_c$. Because $d$ is exactly the additive order of $T_c$ modulo $n$, this total shift is strictly zero:
$$\sigma_{q,c}^{dm}(k) \equiv k + d \cdot T_c \equiv k \pmod n.$$
Since $\sigma_{q,c}^{dm} = k$, the cycle containing $k$ has length $r$, which must be a divisor of $dm$.

To see how $r$ is further constrained, we use the cycle structure of $\sigma_{q,c}^m$. Because it is a pure translation by $T_c$ where $T_c$ has order $d$, the permutation $\sigma_{q,c}^m$ is a product of disjoint cycles with equal length $d$.
By the standard properties of permutations, raising a cycle of length $r$ to the $m$-th power splits that cycle into a product of disjoint cycles of equal length $\frac{r}{\gcd(r, m)}$.
Equating this to our known length $d$, we have:
$$\frac{r}{\gcd(r, m)} = d \implies r = d \cdot \gcd(r, m).$$

The above equation establishes that $r$ is necessarily a multiple of $d$, and also a divisor of $dm$. Let $g = \gcd(r, m)$. Substituting $r = d \cdot g$ into this definition yields $g = \gcd(dg, m)$. Dividing both sides by $g$ reveals that we must have $\gcd(d, m/g) = 1$.

Therefore, rather than a single uniform cycle length in the case when $m=1$, the permutation $\sigma_{q,c}$ can decompose into cycles of varying lengths $dg$ where $g\mid m$. Consequently, the degrees of the irreducible factors of $f(x_0, t)$ must be multiples of $d$ of the form $d \cdot g$, where $g \mid m$ and $\gcd(d, m/g) = 1$.

We summarize these combinatorial constraints into the following theorem.
\begin{thm}[Possible Splitting Types]\label{thm:general_splitting_types}
Let $X$ be an affine normal variety over $\mathbb{F}_q$, and $f \in \mathbb{F}_q[X]$ satisfying the conditions of Theorem \ref{thm:kummer_splitting}. Let $m$ be the multiplicative order of $q$ modulo $n$.

For any $x_0 \in U(\mathbb{F}_q)$, the possible degree $r$ of any irreducible factor of $t^n - f(x_0)$ over $\mathbb{F}_q$ is constrained by $m$. Specifically, let $T_c \equiv c \frac{q^m - 1}{q - 1} \pmod n$ be the translation parameter associated with the Frobenius shift $c \in \mathbb{Z}/n\mathbb{Z}$, where we let $T_c=c$ when $m=1$.
\begin{enumerate}
    \item If $T_c \equiv 0 \pmod n$, all irreducible factors have degrees dividing $m$.
    \item If $T_c \not\equiv 0 \pmod n$, let $d = \frac{n}{\gcd(T_c, n)}$ be the additive order of $T_c$ in $\mathbb{Z}/n\mathbb{Z}$. Then the degree $r$ of any irreducible factor must be of the form $r = d \cdot g$, where $g$ is a divisor of $m$ such that $\gcd(d, m/g) = 1$. In particular, $r$ is always a multiple of $d$ and a divisor of $d \cdot m$.
\end{enumerate}
By enumerating the parameter $c \in \mathbb{Z}/n\mathbb{Z}$, Theorem \ref{thm:kummer_splitting} provides the exact asymptotic distribution of these allowable factorization types.
\end{thm}

Beyond the general factorization shape, we are often specifically interested in whether $t^n - f(x_0)$ has a root in the base field $\mathbb{F}_q$. This happens if and only if the permutation $\sigma_{q,c}$ has a fixed point.

A point $k \in \mathbb{Z}/n\mathbb{Z}$ is fixed by $\sigma_{q,c}$ if and only if:
$$qk + c \equiv k \pmod n \iff (q-1)k \equiv -c \pmod n.$$
This linear congruence has a solution for $k$ if and only if $\gcd(q-1, n)$ divides $c$. When this divisibility condition holds, the polynomial $t^n - f(x_0)$ will have exactly $\gcd(q-1, n)$ rational roots.

More generally, because Theorem \ref{thm:kummer_splitting} dictates that the shift parameter $c$ is uniformly distributed over the $n$ elements of $\mathbb{Z}/n\mathbb{Z}$, the proportion of points yielding an $\mathbb{F}_q$-rational root is exactly the proportion of parameters $c$ that are multiples of $\gcd(q-1, n)$. There are exactly $\frac{n}{\gcd(q-1, n)}$ such multiples in the group.

\begin{thm}[Density of Rational Roots of $t^n-f(x)$]\label{thm:rational_roots_density}
Let $X$ be an affine normal variety over $\mathbb{F}_q$, and $f \in \mathbb{F}_q[X]$ satisfying the conditions of Theorem \ref{thm:kummer_splitting}.

The asymptotic proportion of rational points $x_0 \in U(\mathbb{F}_q)$ for which the polynomial $t^n - f(x_0)$ has at least one root in $\mathbb{F}_q$ is
$$\frac{|\{ x_0 \in U(\mathbb{F}_q) \mid t^n - f(x_0) \text{ has an } \mathbb{F}_q\text{-root} \}|}{|U(\mathbb{F}_q)|} = \frac{1}{\gcd(q-1, n)} + O(q^{-1/2}).$$
\end{thm}

\section{Simultaneous Splitting of Polynomials}\label{sec:simultaneous}
\subsection{Linearly Disjoint Extensions and Product Galois Groups}

Having established the splitting of a single polynomial via Galois theory, we now explore how our framework applies when observing the splitting of multiple polynomials simultaneously. Our result here is a generalization of Cohen's classic work (\cite[Section~3]{Cohen1972distributionII}) to higher dimensional parameter space $X$.

Let $f_1, \dots, f_m \in K[t]$ be separable polynomials over the function field $K = \mathbb{F}_q(X)$. For each $i$, let $L_i$ be the splitting field of $f_i$ over $K$, with Galois group $G_i = \Gal(L_i/K)$. The total splitting field of the product $f_1 \dots f_m$ is the compositum field $L = L_1 \dots L_m$ and we have a natural injective homomorphism:
\begin{align*}
\Gal(L/K) &\hookrightarrow G_1 \times \dots \times G_m \\
\sigma &\mapsto (\sigma|_{L_1}, \dots, \sigma|_{L_m}).
\end{align*}
This map embeds the Galois group of the compositum as a subgroup of the direct product of the individual Galois groups. To understand when this map is a full isomorphism, we require the concept of linear disjointness.

\begin{defn}[Linear Disjointness]\label{def:linear_disjoint}
The fields $L_1, \dots, L_m$ are \textbf{mutually linearly disjoint} over $K$ if the natural $K$-algebra homomorphism from their tensor product to their compositum is an isomorphism:
\begin{align*}
\pi: L_1 \otimes_K \dots \otimes_K L_m &\cong L_1 \dots L_m \\
x_1 \otimes \dots \otimes x_m &\mapsto x_1 \dots x_m
\end{align*}
\end{defn}

When this linear disjointness holds, the tensor product is itself a field, meaning $L \cong L_1 \otimes_K \dots \otimes_K L_m$. Consequently, we have
$$[L : K] = \dim_K L = \dim_K (L_1\otimes_K\cdots\otimes_K L_m) = \prod_{i=1}^m \dim_K L_i = \prod_{i=1}^m[L_i: K].$$
Because the order of the Galois group equals the degree of the extension, the natural injection becomes surjective, yielding the exact isomorphism:
\begin{align*}
\Gal(L/K) &\cong G_1 \times \dots \times G_m \\
\sigma &\mapsto (\sigma|_{L_1}, \dots, \sigma|_{L_m})
\end{align*}

In the specific case where the extensions $L_i/K$ are finite and Galois, this structural independence can be cleanly verified via field intersections (See \cite[Section~14.4]{Dummit_Foote_2004}).
\begin{lem}\label{lem:galois_disjoint_intersection}
Let $L_1, \dots, L_m$ be finite Galois extensions of $K$. They are mutually linearly disjoint over $K$ if and only if for every $2 \le i \le m$, the intersection of $L_i$ with the compositum of the previous fields is exactly the base field:
$$L_i \cap (L_1 \dots L_{i-1}) = K.$$
In particular, two Galois extensions $L_1, L_2$ are linearly disjoint over $K$ if and only if $L_1 \cap L_2 = K$.
\end{lem}

In this linearly disjoint setting, the global Galois action on the roots of $f_i$ is completely decoupled from the roots of $f_j$ for $i \neq j$. Let $R_i$ denote the set of roots of $f_i(t)$ in its splitting field $L_i$. The full set of simultaneous roots for the system of polynomials can be identified with the Cartesian product $R_1 \times \dots \times R_m$. Under the isomorphism $\Gal(L/K) \cong G_1 \times \dots \times G_m$, an element $g = (\sigma_1, \dots, \sigma_m)$ acts diagonally and component-wise on this product set:
$$g \cdot (\alpha_1, \dots, \alpha_m) = (\sigma_1(\alpha_1), \dots, \sigma_m(\alpha_m))$$

This independence transfers directly to the local splitting behavior. We can formalize this relationship between the global Galois action and the local splitting behavior at a rational point $x_0 \in X(\mathbb{F}_q)$ with the following theorem. Let $\bar{R}_{i, x_0}$ be the set of roots of the specialized polynomial $f_i(x_0)$ in $\overline{\mathbb{F}_q}$.

\begin{thm}[Simultaneous Local Frobenius Action]\label{thm:simultaneous_frobenius}
Let $f_1, \dots, f_m \in K[t]$ be separable polynomials such that their respective splitting fields $L_i$ are mutually linearly disjoint over $K$. Let $x_0 \in X(\mathbb{F}_q)$ be a point such that $f_i(x_0)$ all have distinct roots in $\overline{\mathbb{F}_q}$.

Then the action of the local Frobenius $\sigma_q(\alpha) = \alpha^q$ on the product of the specialized root sets $\bar{R}_{1, x_0} \times \dots \times \bar{R}_{m, x_0}$ is isomorphic to the action of $\Frob_{x_0} \in \Gal(L/K)$ on the generic roots set $R_1 \times \dots \times R_m$.

Explicitly, the simultaneous splitting type of the system of polynomials $(f_1(x_0), \dots, f_m(x_0))$ over $\mathbb{F}_q$ is uniquely determined by the cycle structure of the product element:
$$\Frob_{x_0} = (\Frob_{x_0, 1}, \dots, \Frob_{x_0, m}) \in G_1 \times \dots \times G_m.$$
\end{thm}

\begin{proof}
Let $\mathcal{O}$ be the integral closure of the affine coordinate ring containing $x_0$ in the compositum field $L$, and let $\mathfrak{M} \subset \mathcal{O}$ be a maximal ideal lying over $x_0$. Because the polynomials $f_i(x_0)$ all evaluate to having distinct roots in $\overline{\mathbb{F}_q}$, the field extension is unramified at $x_0$.

By Theorem \ref{thm:spliting of polynomials via Galois theory II}(1), for each $1 \le i \le m$, the reduction map modulo $\mathfrak{M}$ yields a perfect bijection between the generic roots $R_i \subset L$ and the specialized local roots $\bar{R}_{i, x_0}$ in the residue field. Taking the Cartesian product of these sets, we obtain a natural bijection:
$$ R_1 \times \dots \times R_m \xrightarrow{\sim} \bar{R}_{1, x_0} \times \dots \times \bar{R}_{m, x_0} $$

By Theorem \ref{thm:spliting of polynomials via Galois theory II}(2), the global Frobenius element $\Frob_{x_0} \in \Gal(L/K)$ acts on these integral elements by exactly mirroring the $q$-th power map $\sigma_q$ on the reduced elements. Thus, for any tuple of generic roots $(\alpha_1, \dots, \alpha_m)$, its reduction modulo $\mathfrak{M}$ satisfies:
$$ \Frob_{x_0}(\alpha_1, \dots, \alpha_m) \equiv (\bar{\alpha}_1^q, \dots, \bar{\alpha}_m^q) = (\sigma_q(\bar{\alpha}_1), \dots, \sigma_q(\bar{\alpha}_m)) \pmod{\mathfrak{M}} $$
This demonstrates that the permutation action of $\Frob_{x_0}$ on the generic roots is perfectly isomorphic to the permutation action of $\sigma_q$ on the specialized roots.

Finally, under the isomorphism $\Gal(L/K) \cong G_1 \times \dots \times G_m$ via $\sigma\mapsto (\sigma|_{L_1}, \ldots, \sigma|_{L_m})$, the restriction of $\Frob_{x_0}$ to each subfield $L_i$ is precisely the individual Frobenius element $\Frob_{x_0, i} \in G_i$. Therefore, $\Frob_{x_0}$ acts diagonally on the roots as the product element $(\Frob_{x_0, 1}, \dots, \Frob_{x_0, m})$. The simultaneous splitting type of $f_1(x_0),\ldots, f_m(x_0)$ over $\mathbb{F}_q$ is thus uniquely determined by the cycle structure of each component of this product element $(\Frob_{x_0, 1}, \dots, \Frob_{x_0, m})$.
\end{proof}

Thus, studying the simultaneous splitting behavior of independent polynomials is equivalent to studying the equidistribution of Frobenius elements in this product group.

\subsection{Application: Multi-$n$-th Power Covers and Multi-distribution of the Character Values}

We now use this geometric framework to generalize a recent result of Slavov \cite{slavov2024square}, which proved the asymptotic independence of square values of several polynomials over a finite field. We extend this to $n$-th power residues, and explain the independence relation more naturally using the simultaneous local Frobenius action developed above.

Let $X$ be an affine normal variety over $\mathbb{F}_q$ of dimension $d$, with function field $K = \mathbb{F}_q(X)$. Let $n \ge 2$ be an integer, and assume $q \equiv 1 \pmod n$ so that $\mathbb{F}_q$ contains the full $n$-th roots of unity, denoted by $\mu_n$. We note that this condition holds trivially in Slavov's $n=2$ case.

Let $f_1,\dots,f_m \in \mathbb{F}_q[X]$ and assume that they are not perfect $d$-th power for all $d\mid n$. We consider the family of polynomials $t^n - f_i(x) \in K[t]$, whose splitting fields over $K$ are the Kummer extensions $L_i = K(\sqrt[n]{f_i})$.

Assume that the classes $[f_1],\dots,[f_m] \in K^\times / (K^\times)^n$ are linearly independent over $\mathbb{Z}/n\mathbb{Z}$. Equivalently, we assume that the only relation of the form
$$\lambda \prod_{i=1}^m f_i^{\varepsilon_i} \in (K^\times)^n, \qquad \lambda \in \mathbb{F}_q^\times, \quad 0 \le \varepsilon_i < n$$
is the trivial one where $\varepsilon_i=0$ for all $i$ and $\lambda$ is a perfect $n$-th power in $\mathbb{F}_q^\times$. We also point that assuming $[f_1],\ldots,[f_m]$ to be linearly independent already forces each $f_i$ to be not a perfect $d$-th power for all $d\mid n$.

The total splitting field of $f_1,\ldots, f_m$ is the compositum field $L = K(\sqrt[n]{f_1},\ldots,\sqrt[n]{f_m})$. We now show that under the above linear independence relation, the splitting field of each $f_i$ are mutually linearly disjoint over $K$, and we have $\Gal(L/K)\cong (\mathbb{Z}/n\mathbb{Z})^m$ as the product of the Galois group of each $L_i$. We first quote Kummer Theory \cite[Chapter~VI, Theorem 8.2]{lang2002algebra}, which will be used in the proof.

\begin{thm}\label{thm:kummer_theory}
    Let $K$ be a field containing a primitive $n$-th root of unity, where $n$ is coprime to the characteristic of $K$. Let $\Delta$ be a finite subgroup of $K^\times / (K^\times)^n$, and let $L = K(\sqrt[n]{\Delta})$ be the field extension obtained by adjoining the $n$-th roots of any set of representatives for $\Delta$.

    Then $L/K$ is a finite abelian Galois extension of exponent dividing $n$. Furthermore, there is a perfect, non-degenerate bilinear pairing:
    $$ \Gal(L/K) \times \Delta \to \mu_n $$
    given by $(\sigma, [a]) \mapsto \sigma(\sqrt[n]{a}) / \sqrt[n]{a}$. Consequently, the degree of the extension is exactly the size of the subgroup, $[L:K] = |\Delta|$, and we have the group isomorphism $\Gal(L/K) \cong \operatorname{Hom}(\Delta, \mu_n)$.
\end{thm}

\begin{prop}
Under the independence assumption, the splitting fields $L_1, \dots, L_m$ are mutually linearly disjoint over $K$, and $\mathbb{F}_q$ is algebraically closed in $L$. Consequently, the Galois group of the compositum is identically $\Gal(L/K) \cong (\mathbb{Z}/n\mathbb{Z})^m$, and its geometric Galois group $G^{\mathrm{geom}}$ is same as the arithmetic Galois group.
\end{prop}
\begin{proof}
This follows from standard Kummer theory over the function field $K$. By Kummer theory, there is a perfect pairing between the Galois group $\Gal(L/K)$ and the subgroup $\Delta \subset K^\times / (K^\times)^n$ generated by the classes $[f_1], \dots, [f_m]$. The degree of the compositum field $[L : K]$ is exactly the size of this subgroup $|\Delta|$. We will show that $|\Delta| = n^m$.

The independence condition implies that no non-trivial product $\prod_{i=1}^m f_i^{\varepsilon_i}$ (where $0 \le \varepsilon_i < n$) belongs to $(K^\times)^n$. This means the classes $[f_1], \dots, [f_m]$ are linearly independent over $\mathbb{Z}/n\mathbb{Z}$, and thus they generate a subgroup of size exactly $|\Delta| = n^m$, whose $n^m$ elements are $\prod_1^mf_i^{e_i}$ with $0\leq e_i<n$. Consequently, $[L : K] = n^m$.

Let $L_i$ be the splitting field of $t^n-f_i$ over $K$, since each Kummer sub-extension is generated by a single $n$-th root, its degree is bounded by $[L_i : K] \le n$. The only way for the compositum of these $m$ fields to reach the maximal degree of $n^m$ is if $[L_i : K] = n$ for all $i$ and the fields are mutually linearly disjoint. Therefore, $L_i$ are mutually linearly disjoint over $K$, each with Galois group $\Gal(L_i/K) = \mathbb{Z}/n\mathbb{Z}$, and we have the isomorphism
$$\Gal(L/K) \cong \Gal(L_1/K) \times \dots \times \Gal(L_m/K) \cong (\mathbb{Z}/n\mathbb{Z})^m.$$

Furthermore, to show that the geometric Galois group $G^{\mathrm{geom}}$ coincides exactly with the full arithmetic Galois group $\Gal(L/K)$, we must prove that $L \cap \overline{\mathbb{F}_q} = \mathbb{F}_q$.

By Kummer theory, any intermediate cyclic field between $K$ and $L$ is generated by the $n$-th root of some product $g = \prod_{i=1}^m f_i^{\varepsilon_i}$. If $L$ contained a non-trivial constant field extension, which is a cylic field extension over $K$, there would exist some element $\alpha \in (L \cap \overline{\mathbb{F}_q}) \setminus \mathbb{F}_q$ such that $K(\alpha) = K(\sqrt[n]{g})$. This implies that $\alpha = \sqrt[n]{g}\cdot h$ for some $h\in K^\times$ and thus $\alpha^n = g \cdot h^n$. Because $\alpha \in \overline{\mathbb{F}_q}$ is a constant, its $n$-th power is a constant scalar $\lambda \in \mathbb{F}_q^\times$. This yields the relation:
$$ \lambda = gh^n = \prod_{i=1}^m f_i^{\varepsilon_i} \cdot h^{n} \implies \lambda \prod_{i=1}^m f_i^{-\varepsilon_i} \in (K^\times)^n. $$
By our independence assumption, we must have $\varepsilon_i = 0$ for all $i$, and $\lambda$ is a perfect $n$-th power in $\mathbb{F}_q^\times$. Recall that $\lambda = \alpha^n$, we must have $\alpha\in\mathbb{F}_q$, and $L\cap\overline{\mathbb{F}_q} = \mathbb{F}_q$. Thus, the geometric Galois group $G^{\mathrm{geom}}$ is exactly $\Gal(L/K)$.
\end{proof}

Let $U \subset X$ be the open dense subvariety where the $f_i$ are non-zero. We know that for $x_0\in U$, each polynomial $t^n-f_i(x_0)$ has distinct roots for all $i$, and the splitting type of $t^n-f_i(x_0)$ over $\mathbb{F}_q$ is governed by the $n$-th power character $\chi_n(f_i(x_0))$, which takes values in $\mu_n$. We can use Theorem \ref{thm:simultaneous_frobenius} and Theorem \ref{thm:nth_power_equidistribution} to map the global Frobenius element $\Frob_{x_0} \in \Gal(L/K)$ directly to these character values.

\begin{lem}
Under the natural identification $\Gal(L/K) \cong \Gal(L_1/K) \times \dots \times \Gal(L_m/K) \cong \mu_n^m$, the $i$-th coordinate of the global Frobenius element $\Frob_{x_0}$ is exactly the $n$-th power character $\chi_n(f_i(x_0))$.
\end{lem}

\begin{proof}
By Theorem \ref{thm:simultaneous_frobenius}, the Frobenius $\Frob_{x_0} \in \Gal(L/K)$ projects to the individual Frobenius elements $\Frob_{x_0, i} \in \Gal(L_i/K)$ in each sub-extension. By definition, $\Frob_{x_0, i}$ acts on the integral closure by raising elements to the $q$-th power modulo a prime ideal $\mathfrak{P}$ lying over $P$. Specifically, its action on the generic root $\sqrt[n]{f_i}$ is:
$$\Frob_{P, i}(\sqrt[n]{f_i}) \equiv (\sqrt[n]{f_i})^q = (\sqrt[n]{f_i})^{q-1} \sqrt[n]{f_i} = (f_i)^{\frac{q-1}{n}} \sqrt[n]{f_i} \pmod{\mathfrak{P}}.$$
Upon evaluating at the point $x_0$ (passing to the residue field $\mathbb{F}_q$), the term $(f_i(x_0))^{\frac{q-1}{n}}$ is precisely the definition of the $n$-th power character $\chi_n(f_i(x_0))$ in $\mathbb{F}_q^\times$. Thus, $\Frob_{P, i}$ acts on the root by multiplying it by $\chi_n(f_i(x_0))$. This identifies the $i$-th coordinate of the product element $\Frob_{x_0}$ exactly with the character value.
\end{proof}

Therefore, prescribing a specific sequence of $n$-th power character values
$$g = (\chi_n(f_1(x_0)),\dots,\chi_n(f_m(x_0))) \in \mu_n^m$$
is completely equivalent to requiring that $\Frob_{x_0} = g$ in $\Gal(L/K)$. Because $\Gal(L/K) \cong (\mathbb{Z}/n\mathbb{Z})^m$ is an abelian group, each element $g$ forms its own singleton conjugacy class ($|C| = 1$). Furthermore, because $\Gal(L/K) = G^{\mathrm{geom}}$, the arithmetic restriction map imposes no missing cosets. Applying the Chebotarev Density Theorem (Theorem \ref{thm:chebotarev}), we immediately obtain the geometric equidistribution of these characters.

\begin{thm}[Equidistribution of $n$-th Power Character]\label{thm:slavov_generalization}
Let $q \equiv 1 \pmod n$, and let $f_1, \dots, f_m \in \mathbb{F}_q(X)^\times$ satisfy the independence condition. Let $g \in \mu_n^m$ be any prescribed pattern of $n$-th power characters. The number of rational points $x_0 \in U(\mathbb{F}_q)$ evaluating to this character pattern is given by:
$$\frac{|\{ x_0 \in U(\mathbb{F}_q) \mid (\chi_n(f_1(x_0)), \dots, \chi_n(f_m(x_0))) = g \}|}{|U(\mathbb{F}_q)| } = \frac{1}{n^m} + O(q^{-1/2}).$$
We also emphasize that this also holds if we replace $U(\mathbb{F}_q)$ by $X(\mathbb{F}_q)$, as $U$ is open and dense in $X$.
\end{thm}

By taking $n=2$, this geometric formulation recovers and explicitly generalizes Theorem 3 of Slavov \cite{slavov2024square}. 

\begin{rmk}
One of the key steps in Slavov's proof is Lemma 11 in \cite{slavov2024square}. In that lemma, Slavov shows that if $f_1,\ldots,f_m$ satisfy the square-independence condition at the begining of this section, then
\[
R=\mathbb F_q[x_1,\ldots,x_r,s_1,\ldots,s_m]/
(f_1-s_1^2,\ldots,f_m-s_m^2)
\]
is an integral domain. We explain how this statement is related to the mutual linear disjointness condition used in Theorem \ref{thm:simultaneous_frobenius}.

Let
\[
    A=\mathbb F_q[x_1,\ldots,x_r],
    \qquad
    K=\operatorname{Frac}(A)=\mathbb F_q(x_1,\ldots,x_r).
\]
For each $i$, consider the double cover
\[
    Y_i=\Spec A[s_i]/(f_i-s_i^2)
\]
of $X = \Spec A$. Imposing all of these equations at once gives the simultaneous cover
\[
    Y=Y_1\times_{X}\cdots\times_{X}Y_m,
\]
whose coordinate ring is
\[
    A[s_1]/(f_1-s_1^2)\otimes_A\cdots\otimes_A A[s_m]/(f_m-s_m^2)
    \cong
    A[s_1,\ldots,s_m]/(f_1-s_1^2,\ldots,f_m-s_m^2)
    =R.
\]

Let $\eta: \Spec K \to X$ be the generic point of $X$. The generic fiber of $Y\to X$ is $$Y_{\eta} = Y\times_X \Spec K,$$ whose coordinate ring is 
\begin{align*}
    R\otimes_A K
    &\cong K[s_1,\ldots,s_m]/(f_1-s_1^2,\ldots,f_m-s_m^2)\\
    &\cong K[s_1]/(f_1-s_1^2)\otimes_K\cdots\otimes_K K[s_m]/(f_m-s_m^2).
\end{align*}
Under the independence condition, those fields $K[s_i]/(f_i-s_i^2)$ are mutually linearly disjoint, and hence their tensor product $R\otimes_A K$ is a field.

Finally, since $R$ is a finite free $A$-module, it is also torsion-free as an $A$-module. Then, the natural map
\[
    R\hookrightarrow R\otimes_A K
\]
is injective. Since $R$ embeds into the field $R\otimes_A K$, the ring $R$ is an integral domain. This recovers the integral-domain conclusion of Slavov's Lemma 11 from the viewpoint of mutual linear disjointness.
\end{rmk}

\section{Geometric and \'Etale Cover Perspective}\label{sec:etale}

Throughout the previous sections, we restricted our attention to affine normal varieties and polynomials with coefficients in the coordinate ring. However, the natural setting for these splitting theorems is considerably broader. In this section, we reframe our main results using the modern language of Grothendieck's \'etale fundamental groups and finite \'etale covers. We will largely state the dictionary between our previous algebraic results and the geometric perspective, refer the details and proofs to Szamuely \cite{szamuely2009galois}.

\subsection{Schemes and Rational Functions}

Let $X$ be a normal, geometrically integral scheme of finite type over $\mathbb{F}_q$ of dimension $d$, with function field $K = \mathbb{F}_q(X)$. Instead of restricting the coefficients to be global regular functions on $X$, we pick $p(t) \in K[t]$ to be a monic, separable polynomial of degree $n$ whose coefficients are rational functions on $X$.

Because the coefficients of $p(t)$ are rational functions, they are regular on some maximal open dense subscheme $U_1 \subset X$. Furthermore, because $p(t)$ is separable over $K$, its discriminant $\Delta(p)$ is a non-zero element of $\mathcal{O}_X(U_1)$. We define our working locus $U$ by removing the vanishing locus of the discriminant:
$$ U = U_1 \setminus V(\Delta(p)). $$
By construction, $U$ is an open dense subscheme of $X$. Since $X$ is irreducible of dimension $d$, the complement $X \setminus U$ is a closed subscheme of dimension at most $d - 1$. By the Lang-Weil bounds, the number of rational points we discarded is bounded by:
$$ |X(\mathbb{F}_q) \setminus U(\mathbb{F}_q)| = O(q^{d-1}). $$
Because the Chebotarev Density Theorem carries an error term of $O(q^{d - 1/2})$, dropping these bad loci does not affect the asymptotic distribution of rational points. Thus, analyzing the splitting behavior of $p(t)$ over the dense open subscheme $U$ is asymptotically equivalent to counting over the entire scheme $X$.

\subsection{The \'Etale Fundamental Group}

Let $\bar{\eta} = \operatorname{Spec}(\overline{K})$ be the geometric generic point of $U$. Because $U$ is connected, the \'etale fundamental group $\pi_1^{\text{\'et}}(U, \bar{\eta})$ classifies all finite \'etale covers of $U$.

Let $L$ be the splitting field of $p(t)$ over $K$. We define $Y$ to be the normalization of $U$ in the field extension $L/K$. Because the discriminant of $p$ is non-zero on $U$, the morphism $Y \to U$ is a finite \'etale Galois cover with group $G = \Gal(L/K)$. By Grothendieck's Galois theory (See Szamuely \cite[Proposition~5.3.8]{szamuely2009galois}), this cover corresponds to a continuous, surjective homomorphism from the \'etale fundamental group to the Galois group:
\begin{equation}\label{eqn: projection from absolute Galois to Galois}
    \rho: \pi_1^{\text{\'et}}(U, \bar{\eta}) \twoheadrightarrow \Gal(L/K)
\end{equation}
Because $\Gal(L/K)$ acts faithfully on the $n$ roots of $p(t)$ in $L$, composing $\rho$ with this action yields a continuous monodromy representation into the symmetric group:
$$ \rho_n: \pi_1^{\text{\'et}}(U, \bar{\eta}) \to S_n. $$

To understand the geometric and arithmetic parts of this action, we consider the base change of $U$ to the algebraic closure, $U_{\overline{\mathbb{F}}_q} = U \times_{\operatorname{Spec} \mathbb{F}_q} \operatorname{Spec} \overline{\mathbb{F}}_q$. The structural morphism $U \to \operatorname{Spec} \mathbb{F}_q$ yields the fundamental exact sequence of \'etale fundamental groups (See Szamuely \cite[Theorem 5.6.1]{szamuely2009galois}):
$$ 1 \to \pi_1^{\text{\'et}}(U_{\overline{\mathbb{F}}_q}, \bar{\eta}) \to \pi_1^{\text{\'et}}(U, \bar{\eta}) \xrightarrow{\pi} \Gal(\overline{\mathbb{F}}_q/\mathbb{F}_q) \to 1. $$
Under the representation $\rho$ in Equation \ref{eqn: projection from absolute Galois to Galois}, the image of the geometric fundamental group $\pi_1^{\text{\'et}}(U_{\overline{\mathbb{F}}_q}, \bar{\eta})$ is exactly the geometric Galois group $G^{\mathrm{geom}} = \Gal(L/K\mathbb{F}_{q^m})$ defined in the previous sections. The quotient map $\pi$ naturally recovers the restriction map to the constant field extension.

\subsection{Frobenius Elements and Splitting Density}

For any $\mathbb{F}_q$-rational point $x_0 \in U(\mathbb{F}_q)$, its residue field is exactly $\kappa(x_0) = \mathbb{F}_q$. The inclusion of this point defines a morphism $\operatorname{Spec} \mathbb{F}_q \hookrightarrow U$. By the functoriality of the \'etale fundamental group, this induces a continuous homomorphism from the absolute Galois group of the finite field
$$ \Gal(\overline{\mathbb{F}}_q/\mathbb{F}_q) \longrightarrow \pi_1^{\text{\'et}}(U, \bar{\eta}),$$
which is well-defined up to inner automorphism (conjugacy).

The canonical generator $\sigma_q(a) = a^q$ of $\Gal(\overline{\mathbb{F}}_q/\mathbb{F}_q)$ therefore lifts to a well-defined conjugacy class in $\pi_1^{\text{\'et}}(U, \bar{\eta})$, which we denote by $\Frob_{x_0}$. By evaluating the representation $\rho_n$ at this Frobenius element, we immediately recover the splitting behavior of the specialized polynomial $p_{x_0}(t)$ (where $p_{x_0}$ denotes the polynomial with coefficients evaluated at $x_0$).

We can now state the main splitting theorem and the Chebotarev density result unified in the language of \'etale fundamental groups.

\begin{thm}[\'Etale Splitting and Density]\label{thm:etale_splitting_density}
Let $X$ be a normal, geometrically integral scheme of finite type over $\mathbb{F}_q$. Let $p(t) \in \mathbb{F}_q(X)[t]$ be a separable polynomial of degree $n$, and let $U \subset X$ be the open subscheme where the coefficients of $p$ are regular and the discriminant is non-vanishing.

Let $\rho_n: \pi_1^{\text{\'et}}(U, \bar{\eta}) \to S_n$ be the monodromy representation corresponding to the action on the roots of $p(t)$.
\begin{enumerate}
    \item For any $x_0 \in U(\mathbb{F}_q)$, the specialized polynomial $p_{x_0}(t) \in \mathbb{F}_q[t]$ has distinct roots in $\overline{\mathbb{F}}_q$. Furthermore, the splitting type of $p_{x_0}(t)$ over $\mathbb{F}_q$ is exactly the cycle type of the permutation $\rho_n(\Frob_{x_0}) \in S_n$.

    \item Let $C \subset S_n$ be a conjugacy class in the image of $\rho_n$ corresponding to a specific cycle type, and assume $C$ lies in the coset mapping to $\sigma_q$. The asymptotic density of rational points exhibiting this splitting type is governed by the Chebotarev Density Theorem over \'etale covers:
    $$ \frac{\#\{ x_0 \in U(\mathbb{F}_q) \mid \rho_n(\Frob_{x_0}) \in C \}}{|U(\mathbb{F}_q)|} = \frac{|C|}{|\operatorname{Im}(\rho_n)^{\mathrm{geom}}|} + O(q^{-1/2}) $$
    where $\operatorname{Im}(\rho_n)^{\mathrm{geom}}$ is the image of the geometric fundamental group $\pi_1^{\text{\'et}}(U_{\overline{\mathbb{F}}_q}, \bar{\eta})$ under $\rho_n$.
\end{enumerate}
\end{thm}

Because $|X(\mathbb{F}_q) \setminus U(\mathbb{F}_q)| = O(q^{d-1})$, we may also replace $U(\mathbb{F}_q)$ above by $X(\mathbb{F}_q)$, and the asymptotic density result remains the same.

\section*{Acknowledgment}
The author received support from NSF Grant DMS 2154223.

\bibliographystyle{plain}
\bibliography{references}

\end{document}